\documentclass{amsart}
\usepackage{amsmath,amsthm,amssymb}
\usepackage{mathtools}
\usepackage{tikz-cd}
\usepackage{graphicx}

\usepackage{times}

\DeclareFontShape{OT1}{cmtt}{m}{n}
     {%
      <5-8.99>cmtt8<9-9.99>cmtt9%
      <10-11.99>cmtt10%
      <12-25>cmtt12%
      }{}
\DeclareMathAlphabet{\mathtt}{OT1}{cmtt}{m}{n}

\definecolor{linkblue}{RGB}{1,1,190}
\definecolor{citered}{RGB}{190,1,1}

\usepackage[linkcolor=linkblue
           ,urlcolor=linkblue
           ,citecolor=citered
           ,colorlinks
           ,bookmarksopen=true,
           ]{hyperref}

\theoremstyle{plain}
\newtheorem{theorem}{Theorem}
\newtheorem{proposition}[theorem]{Proposition}
\newtheorem{lemma}[theorem]{Lemma}
\newtheorem{corollary}[theorem]{Corollary}

\theoremstyle{definition}
\newtheorem{definition}[theorem]{Definition}

\newcommand{\cB}{\mathcal{B}}
\newcommand{\cI}{\mathcal{I}}
\newcommand{\cL}{\mathcal{L}}
\newcommand{\bN}{\mathbb{N}}
\newcommand{\bQ}{\mathbb{Q}}
\newcommand{\bZ}{\mathbb{Z}}

\newcommand{\sL}{\mathsf{L}}

\newcommand{\defit}[1]{\textbf{#1}}

\DeclareMathOperator{\Cl}{Cl}

\DeclareMathOperator{\Max}{Max}

\DeclareMathOperator{\im}{im}

\newcommand{\quo}{\mathbf{Q}}

\DeclarePairedDelimiter{\card}{\lvert}{\rvert}

\newcommand{\zssprod}{\mathbin{\raisebox{0.2ex}{\scalebox{0.6}{$\bullet$}}}}

\author{Daniel Smertnig}
\address{Faculty of Mathematics and Physics (FMF)\\
  University of Ljubljana
  and Institute of Mathematics, Physics and Mechanics (IMFM)\\
  Jadranska ulica 21\\
  1000 Ljubljana, Slovenia}
\email{daniel.smertnig@fmf.uni-lj.si}

\begin{document}

\title{Divide and Transfer: Non-Unique Factorizations Beyond Commutativity}
\markright{Non-Unique Factorizations Beyond Commutativity}

\begin{abstract}
Unique factorization fails in many rings and monoids, but divisor and transfer homomorphisms provide tools to understand non-unique factorizations.
In this expository article, we first explore these notions in the classical setting of commutative Dedekind domains, where monoids of zero-sum sequences appear as a natural combinatorial model.
We then adapt these ideas to the setting of noncommutative Dedekind prime rings using module-theoretic methods.
Going a step further, we discuss Rump and Yang's recent divisor theory for ideals in hereditary noetherian prime rings, where divisors can be visualized in a diagrammatic calculus.
\end{abstract}

\maketitle

\section{Introduction}

Decomposing mathematical objects into their simplest building blocks is a theme almost as old as mathematics itself.
Euclid proved that every natural number factors uniquely into primes.
Kummer discovered that unique factorization fails in rings of cyclotomic integers $\bZ[\zeta]$ in the 19\textsuperscript{th} century, famously collapsing Lamé's attempt at proving Fermat's Last Theorem.
This ultimately led Dedekind to define ideals, a concept now pervasive throughout algebra. Ideals restore unique factorization at a higher level of abstraction in rings of algebraic integers.

Looking beyond numbers, we decompose other objects:
finite groups have unique composition factors by the Jordan-Hölder Theorem, leading to finite simple groups as building blocks of finite groups.
Modules over a ring can often be decomposed into direct sums of indecomposable modules. 
Sometimes, such as over finite-dimensional algebras over a field, the Krull--Remak--Schmidt--Azumaya Theorem guarantees the uniqueness of such decompositions \cite{baeth-wiegand13}.
More often, these decompositions are non-unique.
Permutations factor as products of transpositions, orthogonal matrices as products of reflections.
Singular matrices factor as products of idempotent matrices \cite{erdos67,cossu-tringali24}.

Two basic questions arise: what are the building blocks, and how do they combine to make up the original objects, especially in the non-unique case?
For the second question, the ideas of Kummer and Dedekind provide a blueprint.
By moving to a different level of abstraction, perhaps, we can further break up the simple objects into even smaller pieces, restoring uniqueness.
The study of non-unique factorizations then becomes the study of how these pieces can be assembled to form the original objects.

Divisor and transfer homomorphisms are powerful tools in this endeavor.
Transfer homomorphisms were defined by Halter-Koch, formalizing an idea that Nar\-kie\-wicz used for rings of algebraic integers.
In Section~\ref{sec:commutative}, we will start our journey in this expository article by introducing these notions in the classical setting of commutative Dedekind domains.
Detailed accounts are left to \cite{baginski-chapman11,geroldinger16,geroldinger-zhong20} and the books \cite{geroldinger-halter-koch06,geroldinger09} \cite[Ch.~9]{narkiewicz04}.

In Section~\ref{sec:noncommutative} we will then venture into the noncommutative realm, meeting De\-de\-ki\-nd prime rings as a natural generalization of commutative Dedekind domains.
There, we will see how the ideas can be adapted to a noncommutative setting, based on \cite{smertnig19}.

Concepts from commutative algebra typically have several interesting generalizations to noncommutative algebra.
Hereditary noetherian prime rings (HNP rings) form a strictly larger class than Dedekind prime rings that also coincides with Dedekind domains in the commutative setting.

In Dedekind prime rings, two-sided ideals surprisingly still factor uniquely as \emph{commuting} products of maximal ideals, exactly as in the commutative case.
This is no longer true in HNP rings: these rings allow us a peek into a world in which the factorization of ideals itself fails to be unique, but only ever so slightly.
While HNP rings are classical objects in noncommutative ring theory and integral group representations, the multiplicative structure of their ideals remained a mystery until very recently.

Ten years ago, Rump and Yang introduced a divisor theory that fully describes the structure of ideals in HNP rings \cite{rump-yang16,rump22,rump25}.
The multiplication of ideals is described in terms of the natural noncommutative operation of function composition.
Despite the seemingly abstract setting, pleasingly, these divisors can be visualized as simple diagrams, and ideal multiplication translates into a straightforward gluing of such diagrams.
A glimpse at Rump and Yang's beautiful theory will conclude our short journey in Section~\ref{sec:hnp}.
The visualization is the only minor new contribution of this article.

\subsection{Setting out}
Throughout, we consider the most basic type of factorizations: factorizations of elements of a ring into atoms, also called irreducible elements.
We restrict to non-zero-divisors, where the theory is most fruitful.

If $R$ is a ring, let $R^\bullet$ be its multiplicative monoid of non-zero-divisors.
Most commonly, the ring $R$ will be a domain, then $R^\bullet=R \setminus \{0\}$.
Later it will be convenient to admit some rings with zero-divisors, such as the matrix ring $M_n(\bZ)$. 

A non-invertible $a \in R^\bullet$ is an \defit{atom} if $a = bc$ with $b$,~$c \in R^\bullet$, implies that $b$ or $c$ is invertible. 
A factorization of $r \in R^\bullet$ is then a representation of the form $r=a_1\cdots a_n$ with $a_i$ atoms.
Similar definitions apply if $H$ is any multiplicative monoid.

Let us start with an easy example.
Let $\alpha \coloneqq \frac{1 + \sqrt{-23}}{2}$, and
\[
D \coloneqq \bZ[\alpha] \coloneqq \{\,  a + b \alpha : a, b \in \bZ \,\}.
\]
In this ring, which is the ring of integers of the imaginary quadratic number field $\bQ(\sqrt{-23})$, we have the following factorizations of $8$:
\begin{equation} \label{eq:exm-factor-8}
8 = 2 \cdot 2 \cdot 2 = \frac{3 + \sqrt{-23}}{2} \cdot \frac{3 - \sqrt{-23}}{2}.
\end{equation}

To see that the factors are atoms, we can note that the field norm $N\colon \bQ(\sqrt{-23}) \to \bQ$, mapping $x+y \sqrt{-23}$ to $x^2 + 23y^2$ ($x$,~$y \in \bQ$), restricts to a multiplicative map $D \to \bN_0, $ with $N(a + b \alpha)=a^2 + ab + 6b^2$.
Since $N(2) = 4$ and $N\bigl(\frac{3 \pm \sqrt{-23}}{2}\bigr) = 8$, it suffices to observe that $N$ does not take the value $2$ on $D$.

Note that the two factorizations in \eqref{eq:exm-factor-8} do not even have the same length!
To understand what is going on, we need to consider ideals.

\section{Familiar Landscapes: Commutative Dedekind Domains} \label{sec:commutative}

In a commutative Dedekind domain $R$, every nonzero ideal $I$ factors uniquely into a product of prime ideals \cite[Ch.~9]{atiyah-macdonald69}:
\[
I = P_1^{e_1} P_2^{e_2} \cdots P_n^{e_n}.
\]
In fact, the existence of such factorizations is one of the many equivalent definitions of Dedekind domains \cite[\S V.6]{zariski-samuel58}\cite[\S20.4]{clark-ca}. Rings of integers in number fields, such as the ring $D$ we just encountered, are prototypical examples of Dedekind domains.

Writing $(a)$ for the principal ideal generated by $a \in D$, we can further factor the atoms in our example into two distinct prime ideals $P$ and $Q$:\footnote{We do not discuss how to obtain these factorizations. The interested reader is referred to any introductory text on algebraic number theory.
The actual computation can be performed using a computer algebra system such as SageMath, PARI/GP, or Magma.}
\[
(2) = PQ, \quad \bigg(\frac{3 + \sqrt{-23}}{2}\bigg) = P^3, \quad \bigg(\frac{3 - \sqrt{-23}}{2}\bigg) = Q^3,
\]
and $(8) = P^3 Q^3$.

Here $P$ and $Q$ are non-principal ideals, while $PQ$, $P^3$, and $Q^3$ are principal.
The two different factorizations of $8$ into atoms correspond to different groupings of the prime ideals into products that are principal.
Let us formalize this idea.

Let $\Max(R)$ denote the set of nonzero maximal ideals of $R$.
In a Dedekind domain, these coincide with the nonzero prime ideals.
To multiply two ideals, represented as products of maximal ideals, we simply add up the exponents of the corresponding maximal ideals.
Denoting by $\cI(R)$ the monoid of nonzero ideals, this means that there is an isomorphism
\[
     (\cI(R),\cdot) \to (\bN_0^{(\Max(R))}, +),\quad P_1^{e_1}\cdots P_n^{e_n} \mapsto e_1 P_1 + \cdots + e_n P_n,
\]
with the sum understood to be formal (and not the addition of ideals).
We view elements of $\bN_0^{(\Max(R))}$ as (effective) \defit{divisors}: formal sums of maximal ideals with nonnegative integer coefficients.

Composing the isomorphism with the natural monoid homomorphism $R^\bullet \to \cI(R)$ that maps $a$ to the principal ideal $(a)$, we obtain a monoid homomorphism
\[
    \partial: R^\bullet \to \bN_0^{(\Max(R))}.
\]
For example, we have $\partial(8) = 3P + 3Q$.

It is easily checked that $\partial(a)=\partial(a')$ if and only if $a'=au$ for some invertible $u \in R$, that is, if and only if $a$ and $a'$ are \defit{associates}.
Since factorizations are only interesting up to permutation and associates, the map $\partial$ preserves all relevant factorization-theoretic information.

Divisors reduce multiplication of elements to the addition of nonnegative integer vectors.
To study the factorizations of $a$ through $\partial(a)$ one obstacle remains: we need to be able to recognize which divisors are principal, that is, are in the image of $\partial$.
This leads us to divisor homomorphisms.

Since $\partial$ is a monoid homomorphism, if $a$ divides $b$ (meaning $b=ac$ for some $c$), then $\partial(a)$ divides $\partial(b)$.
Divisibility in the additive monoid $\bN_0^{(\Max(R))}$ means that $\partial(a)$ is a summand of $\partial(b)$.
This is equivalent to $\partial(a) \le \partial(b)$ in the component-wise comparison of the coefficients.
The map $\partial$ also enjoys the converse property.

\begin{definition}
    Let $H$, $D$ be commutative cancellative monoids.
    A homomorphism $\varphi \colon H \to D$ is a \defit{divisor homomorphism} if whenever $\varphi(a)$ divides $\varphi(b)$ in $D$, then $a$ divides $b$ in $H$ (for all $a$,~$b \in H$).
\end{definition}

\begin{proposition}
    The map $\partial \colon R^\bullet \to \bN_0^{(\Max(R))}$ is a divisor homomorphism.
\end{proposition}

\begin{proof}
    Suppose $\partial(a) \le \partial(b)$.
    Then $(b) = P_1 \cdots P_n (a)$ with prime ideals $P_i$.
    This implies $(b) \subseteq (a)$, and so $b = ac$ for some $c \in R^\bullet$.
\end{proof}

The usual pair construction allows us to embed any commutative cancellative monoid $H$ into a group $\quo(H)$, its group of fractions.
Elements of $\quo(H)$ take the form $ab^{-1}$ with $a$,~$b \in H$.

\begin{definition}
    If $\varphi\colon H \to D$ is a divisor homomorphism, then its \defit{class group} is $\Cl(\varphi) \coloneqq \quo(D)/\quo(\im \varphi)$.
\end{definition}

\begin{lemma} \label{l:div-hom-image}
    For $d \in D$, it holds that $d \in \im\varphi$ if and only if $[d]=0$ in $\Cl(\varphi)$.
\end{lemma}

\begin{proof}
    The non-trivial direction is that $[d]=0$ implies $d \in \im \varphi$.
    If $[d]=0$, then $d = \varphi(a) \varphi(b)^{-1}$ for some $a$,~$b \in H$.
    Multiplying, this means $d \varphi(b) = \varphi(a)$.
    Since $\varphi$ is a divisor homomorphism, then $a=cb$ with $c \in H$.
    Now $\varphi(a)=\varphi(c)\varphi(b)=d\varphi(b)$ implies $d = \varphi(c)$ by cancellativity of $D$.
\end{proof}

Applied to $\partial$ we now have a way to recognize principal divisors.
In our example $\Cl(\partial) \cong \bZ/3\bZ = \{\overline{0}, \overline{1}, \overline{2}\}$.
The prime ideals $P$ and $Q$ represent the two non-trivial classes $\overline{1}$ and $\overline{2}$, respectively.
Mapping each prime ideal in our running example to its class, we see that $(2)=PQ$ maps to $\overline 1 \zssprod \overline 2$, a formal product of the two classes (not their sum in the class group!) which we call a sequence.
Similarly,
\[
\bigg(\frac{3 + \sqrt{-23}}{2}\bigg) = P^3 \mapsto \overline 1 \zssprod \overline 1 \zssprod \overline 1, \qquad \bigg(\frac{3 - \sqrt{-23}}{2}\bigg) = Q^3 \mapsto \overline 2 \zssprod \overline 2 \zssprod \overline 2.
\]

Observe that $\overline 1 + \overline 2 = \overline 1 + \overline 1 + \overline 1 = \overline 2 + \overline 2 + \overline 2 = \overline 0$ in $\Cl(\partial)$, reflecting the fact that the corresponding ideals are principal.
Moreover, we can see that no non-trivial subsequence of these three sequences adds up to $\overline{0}$.
This shows that the elements are atoms: for instance, any non-trivial factor of $\frac{3+\sqrt{-23}}{2}$ would give rise to a subsequence of the length three sequence $\overline{1}^3$ whose classes add up to $\overline 0$, but the only non-trivial subsequences of $\overline{1}^3$ are $\overline{1}$ and $\overline{1}^2$, for which this is not true.

The number $8$ maps to $\overline 1^3 \zssprod \overline 2^3$.
The different factorizations of $8$ into atoms in $D$ (up to permutation and associates) correspond precisely to the different factorizations of this sequence into minimal zero-sum sequences:
\[
\overline 1^3 \zssprod \overline 2^3 = (\overline 1 \zssprod \overline 2) \zssprod (\overline 1 \zssprod \overline 2) \zssprod (\overline 1 \zssprod \overline 2) = (\overline 1^3) \zssprod (\overline 2^3).
\]
The combinatorial model now easily confirms that the two factorizations of $8$ are the only possible ones.

We can make this approach rigorous.
Let $(G,+)$ be any abelian group (think of it as the class group), and let $G_0$ be a subset of $G$ (think of it as the set of classes containing prime ideals).
A \defit{sequence} over $G_0$ is a finite collection of elements from $G_0$, where repetition is allowed, and the order of elements does not matter.
We use multiplicative notation for sequences. 
For instance, the sequence $S=g_1 \zssprod g_2^2 \zssprod g_3$ consists of one copy of $g_1$, two copies of $g_2$, and one copy of $g_3$.
Sequences can be multiplied by concatenation.
Formally, the sequences over $G_0$ form the multiplicatively written free abelian monoid with basis $G_0$.

Aside from the formal multiplication of sequences, we can also add up the elements of a given sequence using the group operation of $G$: the \defit{sum} of $S$ as above is $\sigma(S) \coloneqq g_1 + 2g_2 + g_3 \in G$.
A \defit{zero-sum sequence} is one whose sum is $0 \in G$.

\begin{definition}
    The \defit{monoid of zero-sum sequences} over $G_0$, denoted by $\cB(G_0)$, is the monoid of all zero-sum sequences over $G_0$, with the product of two sequences defined to be their concatenation.
\end{definition}

Going back to our example, we have $G=G_0=\bZ/3\bZ$.
Then $\overline{1} \zssprod \overline 2$, $\overline{1}^3$, and $\overline{2}^3$ are zero-sum sequences in $\cB(G_0)$.
The sequence of length one containing the identity element $\overline 0$ is also a zero-sum sequence.
We leave it to the reader to check that these are \emph{all} the \defit{minimal zero-sum sequences}: non-empty zero-sum sequences that contain no non-empty proper zero-sum subsequence.
Every zero-sum sequence can clearly be factored into minimal zero-sum sequences, so the minimal zero-sum sequences are precisely the atoms of $\cB(G_0)$.

In the example we saw that $\cB(\bZ/3\bZ)$ provides a simplified model for factorizations in $D$.
For instance, it tells us about the structure of the factorizations of $8$, but does not allow us to recover the actual factors.
Clearly it is much easier to understand $\cB(\bZ/3\bZ)$ than it is to understand factorizations in $D$ directly.
Let us make this approach precise.

\begin{definition}
    Let $H$ be a cancellative monoid and $D$ be a reduced cancellative monoid (that is, with $1$ as the only invertible element).
    A \defit{transfer homomorphism} is a monoid homomorphism $\theta \colon H \to D$ satisfying the following properties:
    \begin{enumerate}
    \item[(T1)]\label{th:t1} $\theta$ is surjective and if $\theta(a) =1$, then $a$ is invertible in $H$.
    \item[(T2)]\label{th:t2} For all $a \in H$ and all expressions $\theta(a) = d_1d_2$ with $d_1$,~$d_2 \in D$, there exist $b_1$,~$b_2 \in H$ such that $a = b_1 b_2$ and $\theta(b_i) = d_i$.
    \end{enumerate}
\end{definition}

The second property allows us to lift factorizations.
Transfer homomorphisms preserve key factorization data.
For instance, it is easy to check that $a\in H$ is an atom if and only if $\theta(a)$ is an atom in $D$.
We will see some more consequences shortly, but first we observe that our homomorphism $\partial \colon R^\bullet \to \bN_0^{(\Max(R))}$ indeed induces a transfer homomorphism to a monoid of zero-sum sequences.

\begin{theorem} \label{t:transfer}
    If $\varphi\colon H \to D$ is a divisor homomorphism into a free abelian monoid $D=\bN_0^{(P)}$, then it induces a transfer homomorphism $\theta \colon H \to \cB(G_0)$ where $G_0 \coloneqq \{\, [p] : p \in P \,\} \subseteq \Cl(\varphi)$ is the set of classes containing prime divisors.\footnote{A divisor homomorphism, and hence its class group, is not uniquely determined by $H$.
    This can be remedied by further restricting $\varphi$ to be a divisor theory.
    The interested reader can refer to \cite[\S2.4]{geroldinger-halter-koch06}.
    Our map $\partial$ is a divisor theory.}
\end{theorem}

\begin{proof}
    The map $\theta$ is defined as $\pi \circ \varphi$ with $\pi(d_1 + \dots + d_n) = [d_1] \zssprod \cdots \zssprod [d_n]$ for $d_i \in P$.
    For \hyperref[th:t1]{(T1)} first note that $\theta$ is surjective by choice of $G_0$.
    If $\theta(a)=1$, then $\varphi(a) = 0 = \varphi(1)$. 
    Since $\varphi$ is a divisor homomorphism, this means $a$ divides $1$, hence is invertible.

    Let us verify \hyperref[th:t2]{(T2)}.
    Let $\theta(a)=S_1 S_2$ with $S_1$,~$S_2 \in \cB(G_0)$.
    Decompose $\varphi(a) = d_1 + d_2$ with $d_1$,~$d_2 \in D$ mapping to $S_1$,~$S_2$ (the choice may not be unique).
    By construction $[d_1]=\sigma(S_1)=0 \in \Cl(\varphi)$, so $d_1 \in \im\varphi$ by Lemma~\ref{l:div-hom-image}.

    Let $b_1 \in H$ with $\varphi(b_1)=d_1$.
    Then $\varphi(b_1) + d_2 = \varphi(a)$.
    Because $\varphi$ is a divisor homomorphism, now $a = b_1b_2$ for $b_2 \in H$ with $\varphi(b_2)=d_2$.
    Then $\theta(b_2)=S_2$.
\end{proof}

In the case of rings of algebraic integers, such as our initial example $D$, the group $G$ is simply the ideal class group of $D$, which is a finite abelian group, and $G=G_0$ by Chebotarev's Density Theorem \cite[Cor.~2.11.6]{geroldinger-halter-koch06}.

\begin{corollary} \label{c:transfer-ring-of-integers}
    If $R$ is a ring of algebraic integers, then there is a transfer homomorphism $R^\bullet \to \cB(G)$ with $G=\Cl(R)$ the class group.
\end{corollary}

This key result opened the flood gates for studying non-unique factorizations in rings of integers using methods from additive combinatorics, for instance, Ge\-rol\-din\-ger's Structure Theorem for Sets of Lengths \cite{geroldinger88}\cite[\S3]{geroldinger09}. 
The literature is far too vast to cover here, starting points are the standard reference \cite{geroldinger-halter-koch06} and \cite[Ch.~9.5]{tao-vu06}\cite{geroldinger09,grynkiewicz13,geroldinger16,schmid16,geroldinger-zhong20,grynkiewicz22}.
We only look at one of the earliest results and connect it to a current conjecture.

For $a \in H$, the \defit{length set} $\sL(a) \subseteq \bN_0$ consists of all factorization lengths of $a$. In the example $\sL(8)=\{2,3\}$.
The \defit{system of sets of lengths} of $H$ is $\cL(H) = \{\, \sL(a) : a \in H \,\}$.
In the example $\cL(D) = \big\{ \{0\}, \{1\}, \{2,3\}, \dots \big\}$, with $\sL(a)=0$ if $a$ is invertible.
These invariants are preserved by transfer homomorphisms.

\begin{proposition}[{\cite[Lemma 4.2]{geroldinger16}}]
    If $\theta\colon H \to D$ is a transfer homomorphism, then $\sL(a)=\sL(\theta(a))$ for all $a \in H$ and $\cL(H)=\cL(D)$.
\end{proposition}

A domain $R$ is \defit{half-factorial} if every element has a factorization of unique length, that is, if each length set $\sL(a)$ is a singleton.
Carlitz showed that a ring of algebraic integers is half-factorial if and only if $\card{G} \le 2$ \cite{carlitz60,chapman19}.
Using Corollary~\ref{c:transfer-ring-of-integers} this is now an easy exercise: show that $\cB(G)$ is half-factorial if and only if $|G| \le 2$.

An inverse conjecture asks whether $\cL(R^\bullet)$ determines $\Cl(R)$, aside from a few trivial exceptions.
This conjecture has been proven in interesting special cases, such as for products of two cyclic groups, but the general case is wide open \cite{geroldinger-zhong17,geroldinger-schmid19,geroldinger-schmid23}.

As a final observation, we did not need a Dedekind domain for the machinery to work.
All that was needed was a divisor homomorphism into a free abelian monoid. This leads to Krull monoids and Krull domains \cite[Ch.~2]{geroldinger-halter-koch06}\cite{geroldinger09,geroldinger-zhong20}.
This setting not only opens up a wider class of rings to this method, but also allows us to study direct-sum decompositions of modules by the same methods \cite{facchini12,baeth-wiegand13} and \cite{facchini-herbera00,facchini02,facchini-wiegand04,baeth-geroldinger14,baeth-geroldinger-grynkiewicz-smertnig15}.

Instead of going down this road, our eyes are set in a different direction: we step into the noncommutative world.

\section{The Noncommutative Frontier: Dedekind Prime Rings} \label{sec:noncommutative}

A natural noncommutative generalization of Dedekind domains are Dedekind prime rings \cite[Ch.~5]{mcconnell-robson01}\cite{levy00,levy-robson11}.
Before defining them, let us start with two examples.

Beyond matrix rings, the \defit{Hurwitz quaternions} are a classical example:
\[
\mathcal H = \big\{ a+bi + cj + dk \in \mathbb H: a,b,c,d \in \bZ \text{ or } a, b, c, d \in \bZ + \tfrac{1}{2} \big\}.
\]
Here $i^2 = j^2 = k^2 = ijk = -1$.
The ring $\mathcal H$ is Euclidean with respect to its norm $N(a + bi + cj + dk) = a^2 + b^2 + c^2 + d^2$.
As such, it is a principal ideal domain and exhibits a kind of unique factorization \cite[Ch.~5]{conway-smith03}.\footnote{Despite uniqueness, an interesting new phenomenon, called metacommutation, arises \cite{cohn-kumar15}.
This has recently sparked significant activity \cite{forsyth-gurev-shrima16,chari20,babei-chari21,leite-machiavelo25}. The basic concept is already present in an abstract setting in work of Asano and Murata \cite{asano-murata53}.}
We would however rather focus on less-unique factorizations.

Over $\bZ[\sqrt{3}]$, we find a non-principal ideal domain.
In the ring
\[
S \coloneqq \big\{ a + bi + c \tfrac{\sqrt{3}i + j}{2} + d \tfrac{\sqrt{3}+k}{2} : a,b,c,d \in \bZ[\sqrt{3}] \big\},
\]
the element $1-2i+k$ factors as
\[
\begin{split}
1 - 2i + k &= (i+j)\cdot (-1-i-k) = (-1 - j + k) \cdot (i + j) \\
           &= \big(\tfrac{1}{2} - i + \tfrac{\sqrt{3} - 2}{2} k \big) \cdot \big(\tfrac{\sqrt{3}+2}{2} -j + \tfrac{k}{2}\big) = \dots
\end{split}
\]
If we were to compute more examples, we would start to suspect that $S$ is half-factorial.
But how to prove this?
If $S$ were a commutative Dedekind domain, as in the previous section, it would suffice to find the class group $G$ and the set $G_0$.

We will extend this to noncommutative rings by again constructing a transfer homomorphism, but the path will be more arduous.
Several definitions will crop up, but it is not necessary to fully digest them to follow the arguments.

We take a module-theoretic approach, but other routes are possible in special cases \cite{estes91,baeth-ponomarenko-et-al11,smertnig13}.
Let us first consider this in the familiar terrain of $\bZ$.
We can understand the factorization of an integer $n$ through the cyclic group $\bZ/n\bZ$: 
a prime factorization $n = p^{e_1}_1 \cdots p^{e_k}_k$ corresponds to a composition series of $\bZ/n\bZ$ with the composition factor $\bZ/p_i \bZ$ repeated $e_i$ times.
Unique factorization of integers thus follows from the Jordan-Hölder Theorem.

The same works in commutative Dedekind domains.
If $I=P_1^{e_1} \cdots P_k^{e_k}$ is the factorization of a nonzero ideal, then $R/I$ is a finite-length module and its composition factors are $R/P_i$ repeated $e_i$ times.
In this way, we can obtain the divisor theory $\partial$ from Section~\ref{sec:commutative} from the composition factors of $R/I$.

For this to work in noncommutative rings, we need a suitable replacement for Dedekind domains.
In particular, each module $R/aR$ for $a \in R^\bullet$ has to have finite length.
We could require $R$ to be a principal ideal domain, but this is too restrictive, as it excludes the example $S$.

In a principal ideal domain, every left and every right ideal is free as left, respectively right, module.
Let us replace free modules by something more general.
A (right) module $P$ is \textbf{projective} if it is a direct summand of a free right module \cite[\S2A]{lam99}, that is, if $P \oplus X \cong R^{(I)}$ for some set $I$.
A ring $R$ is \defit{hereditary} if every right and left ideal is projective.
Then submodules of projective modules are again projective \cite[\S2E]{lam99}.

Commutative hereditary domains are precisely the Dedekind domains \cite[Exm. 2.32(k)]{lam99}, so this is promising.
Unfortunately, noncommutative hereditary domains are less well-behaved.
They include rings such as the free algebra $\mathbb C\langle x, y \rangle$, which do not have the desired finite-length property for $R/aR$.
Indeed, while a commutative hereditary domain is always noetherian and integrally closed, for noncommutative domains this is no longer true.
We have to additionally impose these conditions.

\begin{definition}
    A \defit{hereditary noetherian prime ring} (\defit{HNP ring}) is a hereditary ring $R$ that is both left and right noetherian, and prime (meaning $IJ \ne 0$ for nonzero ideals $I$,~$J$ and $R$ is not the zero ring).
\end{definition}

While we added noetherianity, we weakened \emph{domain} to \emph{prime}.
This admits rings like $M_n(\bZ)$ that are not domains.
Let $R$ be a HNP ring. We now have the following.

\begin{proposition}[{\cite[Prop.~5.4.5]{mcconnell-robson01}}] \label{p:finite-length}
    The quotient $R/aR$ has finite length for $a \in R^\bullet$.
\end{proposition}

To proceed with the construction of a transfer homomorphism, we add an extra condition that replaces integral closedness.
While the results hold without it \cite{smertnig19}, the arguments are much more technical.
For this, we need a ring of fractions $\quo(R)$ of $R$.
Forming rings of fractions of noncommutative rings is a subtle affair \cite{cohn71}\cite[\S9]{lam99}, but in our setting Goldie's Theorem \cite[Ch.~2\,\&\,3]{mcconnell-robson01}\cite[\S11]{lam99} shows that there exists a ring of fractions $\quo(R) = \{\, ab^{-1} : a \in R,\ b \in R^\bullet \,\}$ and $\quo(R)$ is simply a matrix ring over a division ring.
Now we arrive at our replacement for Dedekind domains.

\begin{definition}
    A \defit{Dedekind prime ring} is a \textbf{hereditary noetherian prime} (\textbf{HNP}) ring such that $\{\, x \in \quo(R) : xI \subseteq I \,\} = R$ for every nonzero ideal $I$ of $R$.
\end{definition}

HNP rings and Dedekind prime rings may seem arcane at first, but are quite natural: noetherian prime rings are the noncommutative analogues of commutative noetherian domains.
Being hereditary just means that the global dimension, a natural homological dimension, is at most one \cite[Prop.~5.14]{lam99}.
So we are simply looking at the smallest interesting class of noncommutative noetherian rings.

The examples $\mathcal H$ and $S$ are Dedekind prime rings.
They fall into a more general class of maximal orders in central simple algebras over number fields.
These rings are noncommutative analogues of rings of algebraic integers and appear in the integral representation theory of groups \cite{reiner75,curtis-reiner81,curtis-reiner87}.
Other examples are certain endomorphism rings of modules \cite[Prop.~5.3.15]{mcconnell-robson01}.

Now let $R$ be a Dedekind prime ring, with $R \ne \quo(R)$ to avoid trivialities.
Following the example of the integers, we can now define a replacement for the divisor theory $\partial$ from Section~\ref{sec:commutative}.
Let $\mathcal S(R)$ be the set of isomorphism classes of simple right $R$-modules.
We view the composition factors of a finite-length module $M$ as an element $(M)$ of the free abelian monoid $\bN_0^{(\mathcal S(R))}$.
For instance, if $M$ has composition factors $V$, $W$, $V$, then $(M) = 2(V) + (W)$.
Then $(R/I)$ for a right ideal $I$, with $I \cap R^\bullet \ne \emptyset$, plays the role of the divisor $\partial(I)$.

In our example, we have $(S/(1-2i+k)S)=2(V) + 2(W)$ for non-isomorphic simple modules $V$ and $W$,
and
\[
\begin{split}
&(S/(i+j)S)=2(V), \quad (S/(-1 - i - k)S)=(S/(-1 - j + k)S)=2(W),\\
&(S/(\tfrac{1}{2} - i+ \tfrac{\sqrt{3} - 2}{2} k)S)= (S/(\tfrac{\sqrt{3}+2}{2} -j + \tfrac{k}{2})S) = (V)+(W).
\end{split}
\]

Crucial in the commutative case was that we can recognize principal ideals from their divisors.
We need the same here: if $I=aR$ for some $a \in R^\bullet$, we need to be able to tell this purely from $(R/I)$.
Here $I = aR$ for some $a \in R^\bullet$ if and only if $I$ is free, by \cite[Prop.~3.1.15(iv)]{mcconnell-robson01}.
So we need to understand whether $I$ is free from $(R/I)$.

Let us first consider something less ambitious: what can we recover from $R/I \cong R/bR$ with $b \in R^\bullet$?
Schanuel's Lemma \cite[Lemma~5.1]{lam99} proves useful.
It allows us to cross-multiply isomorphic quotients, reminiscent of how we calculate with fractions.

\begin{lemma} \label{l:schanuel}
    Let $M$,~$M'$, $N$, $N'$ be modules with $M$, $N$ projective. If
    \[
    M/M' \cong N/N',\quad\text{then}\quad M \oplus N' \cong N \oplus M'.
    \]
\end{lemma}

\begin{proof}
    Applying Schanuel's Lemma to the short exact sequences
    \[
    \begin{tikzcd}
    & 0 \arrow{r} & M' \arrow{r}  & M \arrow{r} & M/M' \arrow{r} \arrow{d}{\rotatebox{90}{\(\sim\)}} & 0\phantom{.} \\
    & 0 \arrow{r} & N' \arrow{r} & N \arrow{r} & N/N' \arrow{r} & 0
    \end{tikzcd} 
    \]
    shows the claim.
\end{proof}

From $R/I \cong R/bR$ we obtain $R \oplus I \cong R \oplus bR$.
This is not enough to conclude $I \cong bR$.
For now, we sidestep this with a new definition, weaker than isomorphism.

\begin{definition}
    Two finitely generated projective modules $M$,~$N$ are \defit{stably isomorphic} if there exists some $k \ge 0$ such that $M \oplus R^k \cong N \oplus R^k$. Equivalently, there exists a finitely generated projective module $X$ such that $M \oplus X \cong N \oplus X$.
\end{definition}

Now we can improve the previous lemma to the case where only $(R/I)=(R/bR)$.

\begin{proposition} \label{p:comp-series-iso}
    Let $M$,~$M'$, $N$, $N'$ be finitely generated projective modules such that $M/M'$ and $N/N'$ are finite-length modules with $(M/M') = (N/N')$.
    Then $M \oplus N'$ and $N \oplus M'$ are stably isomorphic.
\end{proposition}

\begin{proof}
    Let $M=M_0 \supsetneq M_1 \supsetneq \cdots \supsetneq M_k = M'$ and $N=N_0 \supsetneq N_1 \supsetneq \cdots \supsetneq N_k = N'$ be composition series.
    By the Jordan-Hölder Theorem, there exists a permutation $\sigma$ such that $M_{i-1}/M_i \cong N_{\sigma(i)-1}/N_{\sigma(i)}$ for all $i$.
    Because $R$ is hereditary, all modules $M_i$ and $N_i$ are finitely generated projective.
    Lemma~\ref{l:schanuel} shows
    \[
    M_{i-1} \oplus N_{\sigma(i)} \cong N_{\sigma(i)-1} \oplus M_i \qquad\text{for all $i$.}
    \]
    Taking the direct sum over $1 \le i \le k$, we obtain
    \[
    M_0 \oplus X \oplus N_k \cong N_0 \oplus X \oplus M_k,
    \]
    with $X \coloneqq \bigoplus_{i=1}^{k-1} N_i \oplus M_i$ finitely generated projective.
\end{proof}

While isomorphic modules are stably isomorphic, the converse can fail.
Luckily, in Dedekind prime rings, stable isomorphism nearly coincides with isomorphism, as a consequence of the deep Cancellation Theorem of Stafford \cite[Thm.~11.7.13]{mcconnell-robson01}.%
\footnote{This theorem even allows us to cancel $X$ in the proof of Proposition~\ref{p:comp-series-iso}, and to conclude $M \oplus N' \cong N \oplus M'$ \cite[Thm.~34.6]{levy-robson11}, just as in Lemma~\ref{l:schanuel}.}

\begin{theorem}[{\cite[Cor.~11.7.14]{mcconnell-robson01}}] \label{t:cancellation}
    Two finitely generated projective modules $M$, $N$ are stably isomorphic if and only if $M \oplus R \cong N \oplus R$.
\end{theorem}

Thus, we can cancel all but one copy of $R$ from a stable isomorphism.
Unfortunately, the last copy need not cancel.
This leads to an additional condition, which always holds if $R$ is commutative (by \cite[Prop.~11.1.16]{mcconnell-robson01} or the Structure Theorem for finitely generated modules over Dedekind domains \cite[\S C-5.3]{rotman17}).

\begin{definition}
    The ring $R$ is \defit{Hermite} if $M \oplus R \cong R \oplus R$ implies $M \cong R$ for all finitely generated projective modules $M$.
\end{definition}

With the Hermite condition, we can finally recognize free right ideals $I$ from $(R/I)$, as now $(R/I) = (R/bR)$ implies $I \oplus R \cong bR \oplus R$ and hence $I \cong R$.
To proceed, we need one more observation: every ``divisor'' actually comes from a right ideal.

\begin{lemma}
    Every $(M) \in \bN_0^{(\mathcal S(R))}$ occurs as $(M)=(R/I)$ for some right ideal $I$.
\end{lemma}

\begin{proof}
    Every nonzero right ideal $J$ of a Dedekind prime ring is a generator module \cite[Thm.~5.2.10]{mcconnell-robson01} (see \cite[\S18]{lam99} for generators).
    This means, for every simple module $V$, there exists $I \subseteq J$ with $J/I \cong V$.
    Since we excluded the trivial case $R = \quo(R)$, the ring $R$ has no minimal right ideals \cite[Lemma~3.3.4]{mcconnell-robson01}.
    Thus, also $I \ne 0$, and we can finish by induction on the length of $(M)$.
\end{proof}

Putting everything together, we have a replacement for the principal divisors.

\begin{theorem}[{\cite[Prop.~3.17]{smertnig19}}] \label{t:nc-div-hom}
    Let $R$ be a Hermite Dedekind prime ring and let
    \[
    H \coloneqq \big\{\, (R/aR) : a \in R^\bullet \,\big\} \subseteq \bN_0^{(\mathcal S(R))}.
    \]
    Then $H$ is a monoid and the inclusion $H \hookrightarrow \bN_0^{(\mathcal S(R))}$ is a divisor homomorphism.
\end{theorem}

\begin{proof}
    We see that $H$ is a monoid because $(R/abR)=(R/aR)+(aR/abR)=(R/aR)+(R/bR)$, using $R/bR \cong aR/abR$ via $x+bR \mapsto ax + abR$.

    Now it suffices to show that $(R/I) + (R/aR) = (R/bR)$ with $a$,~$b \in R^\bullet$ implies that $I$ is free, that is, isomorphic to $R$.
    Because $R$ is Hermite, it suffices to show that $I$ is stably isomorphic to $R$.
    For this, note $\big((R \oplus R)/(I \oplus aR)\big) = (R/I) + (R/aR) = (R/bR)$, so Proposition~\ref{p:comp-series-iso} implies $R \oplus R \oplus bR \cong I \oplus aR \oplus R$.
\end{proof}

Theorem~\ref{t:transfer} now gives a transfer homomorphism for $H$!
One more step is needed to get to $R^\bullet$ instead of $H$.
The Jordan-Hölder Theorem guarantees uniqueness of the composition factors up to permutation.
It does not say which permutations of the factors can actually be realized.

Fixing this requires a final condition, tying $R$ closer to the commutative world.
The ring $R$ is \defit{bounded} if every one-sided ideal of the form $aR$ or $Ra$ with $a \in R^\bullet$ contains a nonzero two-sided ideal.
The rings $\mathcal H$ and $S$ above are bounded.
By contrast, the first Weyl algebra $A_1 \coloneqq \mathbb C\langle x, y \,|\, yx - xy = 1 \rangle$ is a Dedekind prime ring but not bounded.

\begin{proposition} \label{p:permute-composition-factors}
    If $R$ is a bounded Dedekind prime ring and $M$ is a finite-length module, then every permutation of composition factors of $M$ can be realized in a composition series of $M$.
\end{proposition}

\begin{proof}
    By induction on the length it suffices to check this for modules of length two.
    Let $M$ be a module with a submodule $N$ such that $M/N \cong V$ and $N\cong W$ are non-isomorphic simple modules $V$,~$W$.
    Then there is a short exact sequence
    \[
    \begin{tikzcd}
    0 \arrow{r} & N \arrow{r} & M \arrow{r} & M/N \arrow{r} & 0.
    \end{tikzcd}
    \]
    In \emph{bounded} Dedekind prime rings any such sequence splits.
    This follows from \cite[Thm.~15.4]{levy-robson11}, which applies because of \cite[Lemma 14.3]{levy-robson11}.
    So $M = N \oplus N'$ with $N' \cong M/N$, and $M$ also has a composition series with composition factors transposed.
\end{proof}

Finally, we obtain the main result, a transfer homomorphism for $R^\bullet$.

\begin{theorem}[{\cite[Thm.~4.4]{smertnig19}}]
    If $R$ is a bounded Hermite Dedekind prime ring, then there exists a transfer homomorphism from $R^\bullet$ to a monoid of zero-sum sequences.
\end{theorem}

\begin{proof}
    It suffices to show that $\theta\colon R^\bullet \to H$, $a \mapsto (R/aR)$, with $H$ as in Theorem~\ref{t:nc-div-hom}, is a transfer homomorphism.
    The non-trivial part is that $\theta(a)=(R/d_1R) + (R/d_2R)$ with $d_i \in R^\bullet$ implies $a=b_1b_2$ with $b_i \in R^\bullet$ such that $(R/b_iR)=(R/d_iR)$.

    Rearranging composition factors of $R/aR$, we find a right ideal $R \supseteq I \supseteq aR$ with $(R/I)=(R/d_1R)$.
    As in Theorem~\ref{t:nc-div-hom}, then $I=b_1R$ with $b_1 \in R^\bullet$.
    Because $aR \subseteq b_1R$, there exists $b_2 \in R^\bullet$ with $a=b_1b_2$ and $(R/b_2R)=(I/aR)=(R/d_2R)$.
\end{proof}

A nice thing about this result is that the target monoids are the same ones as in the commutative case!
We can therefore make use of the vast literature on zero-sum sequences over abelian groups even in the noncommutative setting.

The ring $S$ from the beginning of this section is Hermite and has class group $\bZ/2\bZ$ \cite[Table 2]{smertnig-voight19}, with the two simple modules $V$ and $W$ from $(S/(1-2i+k)S)$ both representing the non-trivial class.
So we finally conclude that $S$ is half-factorial.

Unfortunately, dropping either the boundedness condition or the Hermite condition makes the conclusion fail catastrophically \cite[Thm.~1.2]{smertnig13}\cite[\S 6]{smertnig19}.
However, in the classical setting of maximal orders in central simple algebras over number fields, the boundedness condition is always satisfied.
By a deep result of Eichler, the Hermite condition can only fail in the case of quaternion algebras, where the failure has been completely classified \cite{smertnig-voight19}.
In fact, here the Hermite condition is \emph{necessary} for the existence of a transfer homomorphism to a monoid of zero-sum sequences, and its absence means that factorizations behave quite wildly \cite[Thm.~1.2]{smertnig13}.

The methods generalize to bounded Hermite HNP rings \cite{smertnig19}.
However, let us instead next turn our attention from elements to two-sided ideals.

\section{Uncharted Peaks: Ideals in HNP rings} \label{sec:hnp}

In a Dedekind prime ring, every nonzero ideal factors uniquely as a product of maximal ideals \cite[Thm.~5.2.9]{mcconnell-robson01}.
For two distinct maximal ideals $M$,~$N$ we have $MN = M \cap N = NM$, so this multiplication is even commutative, exactly as in the commutative case.

Let $R$ be an arbitrary HNP ring with $R \ne \quo(R)$.
The following ring $T_1$ is the easiest example of such a ring that is not Dedekind.
For a prime $p$, let
\[
T_1 \coloneqq \begin{bmatrix}
    \bZ & (p) \\
    \bZ & \bZ
\end{bmatrix} \qquad\text{and}\qquad M \coloneqq \begin{bmatrix} (p) & (p) \\ \bZ & \bZ \end{bmatrix}.
\]
Here $M \in \Max(T_1)$ is idempotent, so the ideals cannot form a free abelian monoid.

Our goal is to understand the multiplicative structure of the ideals.
For a more prototypical example, we consider a slightly bigger ring $T_2$.
Let $D$ be a discrete valuation ring with prime element $\pi$.
For instance, take $D = \bZ_{(p)} = \{\, \frac{a}{b} : a,b \in \bZ,\ p \nmid b \,\}$, the localization of $\bZ$ at the prime ideal $(p)$, with $\pi = p$.
Let
\[T_2 \coloneqq \begin{bmatrix}
        D & (\pi) & (\pi) \\
        D & D & (\pi) \\
        D & D & D \end{bmatrix}
\qquad\text{and}\qquad
J(T_2) \coloneqq \begin{bmatrix}
        (\pi) & (\pi) & (\pi) \\
        D & (\pi) & (\pi) \\
        D & D & (\pi) \end{bmatrix}.
\]
Observing that $1-A$ is invertible in $T_2$ for $A \in J(T_2)$, together with $T_2/J(T_2) \cong (D/(\pi))^3$, shows that $J(T_2)$ is the Jacobson radical of $T_2$.
The ring $T_2$ thus has three maximal ideals:
\[
Q_1 = \begin{bmatrix}
        D & (\pi) & (\pi) \\
        D & D & (\pi) \\
        D & D & (\pi) \end{bmatrix},\quad 
Q_2 = \begin{bmatrix}
        D & (\pi) & (\pi) \\
        D & (\pi) & (\pi) \\
        D & D & D \end{bmatrix},\quad
Q_3 = \begin{bmatrix}
        (\pi) & (\pi) & (\pi) \\
        D & D & (\pi) \\
        D & D & D \end{bmatrix}.
\]

We shall see how the pictures in Figures \ref{fig:some-divisors} and \ref{fig:divisor-factorization} represent these ideals and their multiplication; we will not have enough space to answer the \emph{why}.

First we need to understand the set of all nonzero ideals.
Before, it was enough to express an ideal as a product of maximal ideals and note the multiplicity of each maximal ideal, leading to the classical divisor theory $\partial$ in Section~\ref{sec:commutative}.
Now
\[
Q_1 Q_2 = \begin{bmatrix}
        D & (\pi) & (\pi) \\
        D & (\pi) & (\pi) \\
        D & (\pi) & (\pi)
\end{bmatrix},
\quad\text{but}\quad
Q_2 Q_1 = \begin{bmatrix}
    D & (\pi) & (\pi) \\
    D & (\pi) & (\pi) \\
    D & D & (\pi)
\end{bmatrix} = Q_1 \cap Q_2.
\]
So $Q_1 Q_2 \ne Q_2 Q_1$, and the previous approach is clearly no longer sufficient.
One can further compute $Q_1 Q_2 = Q_1 Q_2 Q_1$, so the monoid of ideals is not even cancellative.

To recover a good notion of divisors, we take a hint from Section~\ref{sec:noncommutative} and consider composition series.
Only this time, we do not consider a maximal chain of right ideals between $R$ and an ideal $I$, but a maximal chain of two-sided ideals.

\begin{proposition}[{\cite[\S5]{rump-yang16}}]
    Let $I$ be a nonzero ideal of $R$, and let
    \[
    R = I_0 \supsetneq I_1 \supsetneq I_2 \supsetneq \cdots \supsetneq I_n = I
    \]
    be a maximal chain of ideals.
    \begin{enumerate}
        \item For each $k$, there exists a unique maximal ideal $P_k$ of $R$ such that $P_k I_{k-1} \subseteq I_k$.
        \item The ideals $P_1$, \dots,~$P_n$ are uniquely determined by $I$, up to permutation.
    \end{enumerate}
\end{proposition}

This leads to Rump and Yang's definition of divisors.

\begin{definition}[{\cite{rump-yang16,rump25}}]
    The \textbf{divisor} $\partial(I)$ of a nonzero ideal is the formal sum
    \[
        \partial(I) = P_1 + \cdots + P_n \in \bN_0^{(\Max(R))}
    \]
    with $P_1$, \dots, $P_n$ as in the previous proposition.
\end{definition}

In the example, we have $\partial(Q_i) = Q_i$ and $\partial(J(T_2)) = Q_1 + Q_2 + Q_3$.
More interestingly, we have $\partial(Q_2Q_1) = Q_1 + Q_2$ while $\partial(Q_1 Q_2) = 2Q_1 + Q_2$.

To describe ideals through their associated divisors, the map $\partial$ needs to be injective.
The argument is intricate, and we omit it: it proceeds by showing that the lattice of ideals is distributive, meaning $I \cap (J+K) = (I \cap J) + (I \cap K)$ for all $I$,~$J$,~$K$.

\begin{theorem}[{\cite[Prop.~17]{rump-yang16}}]
    The map $\partial$ from nonzero ideals to $\bN_0^{(\Max(R))}$ is injective.
\end{theorem}

Thus, in principle, we should be able to describe the multiplication of ideals using their divisors: given ideals $I$ and $J$ with divisors $D=\partial(I)$ and $E=\partial(J)$, what is $\partial(IJ)$ in terms of $D$ and $E$? Looking at our example $Q_1 Q_2 \ne Q_2 Q_1$ above, we see that the answer is not simply $D+E$.
Indeed, adding divisors is a commutative operation, while ideal multiplication is noncommutative.

We first need a better understanding of the set of maximal ideals $\Max(R)$.
There is a natural bijection $\tau$ on the set of nonzero ideals, although the precise definition will not be important for us:
\[
    I \mapsto \tau(I) \coloneqq (R\! :_l\! (R\! :_l\! I))\quad\text{with}\quad (R\!:_l\! I) \coloneqq \{\, x \in \quo(R) : xI \subseteq R \,\}.
\]
The map $\tau$ restricts to a bijection on maximal ideals, and has finite orbits \cite[\S 5]{rump-yang16}.
This partitions $\Max(R)$ into disjoint finite \textbf{cycles} of maximal ideals, with each cycle having a cyclic order induced by the action of $\tau$.

In the example above, there is just one cycle which has length $3$: $\tau(Q_1)=Q_2$, $\tau(Q_2)=Q_3$, and $\tau(Q_3)=Q_1$.
In general, arbitrary cycle structures are possible \cite[\S 25]{levy-robson11}. For instance, the ring
\[
\begin{bmatrix}
    \bZ & (p) & (pq) \\
    \bZ & \bZ & (pq) \\
    \bZ & \bZ & \bZ
\end{bmatrix}
\]
with $p \ne q$ primes, has one $2$-cycle (over $q$), one $3$-cycle (over $p$), and countably many $1$-cycles (over the remaining primes; these are uninteresting).
It turns out that disjoint cycles behave independent of each other, so our example of a single $3$-cycle is already sufficiently general to understand the general situation.

We now explain how to view divisors as functions.

\begin{definition}
    Let $\widetilde{\Max(R)} \coloneqq \Max(R) \times \bZ$.
\end{definition}

We view this as a covering space of $\Max(R)$: the point $(P,n)$ lies above $P$ for all $n$.
If $C = \{ P_1, \dots, P_l \}$ is a cycle of maximal ideals, the corresponding set $C \times \bZ$ is totally ordered lexicographically: $(P_i, n) \le (P_j, m)$ if $n < m$ or if $n = m$ and $i \le j$.

\begin{definition}[{\cite[\S 8]{rump-yang16}\cite[\S 5]{rump25}}]
    For $D \in \bN_0^{(\Max(R))}$, let $\widehat D \colon \widetilde{\Max(R)} \to \widetilde{\Max(R)}$ be the function that moves each point $(P,n)$ forward by $D(P)$ steps in the natural order.
\end{definition}

For example, for the divisor $D=2Q_1 + Q_3$,
\[
    \widehat D(Q_1,n) = (Q_3,n), \quad \widehat D(Q_2,n) = (Q_2,n), \quad \widehat D(Q_3,n) = (Q_1,n+1).
\]
This embeds divisors into the set of self-maps of $\widetilde{\Max(R)}$.

\begin{figure}
    \centering
    \includegraphics{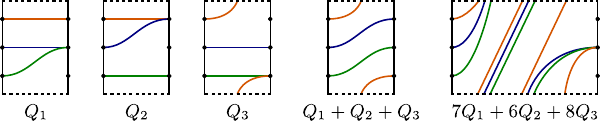}
\caption{Visualization of the divisors corresponding to the idempotent maximal ideals $Q_1$, $Q_2$, $Q_3$, the maximal invertible ideal $J(T_2)$ with $\partial(J(T_2)) = Q_1 + Q_2 + Q_3$, and the more complicated divisor $7Q_1 + 6Q_2 + 8Q_3$. (Colors  for easier readability.)} \label{fig:some-divisors}
\end{figure}
The covering space idea gives a way of visualizing such functions.
Think of a cylinder $[0,1] \times S^1$ with $l$ marked points $P_1, \dots, P_l$, in cyclic order, on each of the two boundary circles.
Now $\widehat D$ can be represented by drawing $l$ paths, connecting each point $P_i$ on the left circle to some point on the right circle, moving forward $D(P_i)$ steps.
Paths only move forward in the cyclic order and may wind around the cylinder multiple times, as illustrated in Figure~\ref{fig:some-divisors} (the paths are understood up to homotopy).

For instance, if $D(Q_1)=4$, then the path moves forward $4$ steps in the cyclic order. Thus, it winds around the cylinder once ($3$ steps), then connects to $Q_2$.
If $D(Q_1)=6$, the path winds around twice before connecting $Q_1$ to itself.

The information that $\widehat D$ contains is: (i) which points on the left side connect to which points on the right and (ii) how many times each path winds around the cylinder.

Self-maps on a set can be composed by function composition, so this gives us a way to compose divisors!
To match ideal multiplication, the order of composition needs to be reversed (or think of functions as composing from left to right).

\begin{definition}
    The \defit{composition} of two divisors $D$, $E$ is the unique divisor $D \circ E$ with $\widehat{D \circ E} = \widehat E \circ \widehat D$.
\end{definition}

\begin{figure}
    \centering
\includegraphics{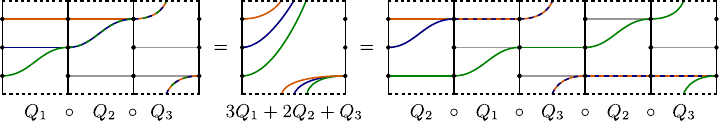}
\caption{The divisor $D=3Q_1+2 Q_2 + Q_3$ factored in two different ways, as $D=Q_1 \circ Q_2 \circ Q_3$ and $D=Q_2 \circ Q_1 \circ Q_3 \circ Q_2 \circ Q_3$.}\label{fig:divisor-factorization}
\end{figure}
Visually, this corresponds to the natural gluing of two diagrams, as in Figure~\ref{fig:divisor-factorization}.
From the diagram it is immediate that each $Q_i$ is idempotent.
In this visualization, it is also easy to observe $Q_1 \circ Q_2 = Q_1 \circ Q_2 \circ Q_1$.
Figure~\ref{fig:divisor-factorization} shows a slightly more complicated example, factoring $D=3Q_1 + 2 Q_2 + Q_3$ in two different ways.

Rump and Yang showed that this composition corresponds to ideal multiplication.
\begin{theorem}[{\cite[Thm.~5]{rump-yang16}, \cite[Thm.~4]{rump25}}]
    For all nonzero ideals $I$,~$J$ of $R$, it holds that $\partial(IJ) = \partial(I) \circ \partial(J)$.
    In particular, the map $\partial$ is an injective homomorphism from nonzero ideals to $(\bN_0^{(\Max(R))}, \circ)$.
\end{theorem}

The reader may have noticed that in the diagrams paths can merge but never cross.
Having such a representation precisely characterizes the image of $\partial$.
Algebraically it can be stated as follows.

\begin{theorem}[{\cite[Thm.~5]{rump25}}]
    A divisor $D$ is in the image of $\partial$ if and only if $D(\tau(P)) \ge D(P) - 1$ for all $P \in \Max(R)$.
\end{theorem}

A consequence is the following surprisingly non-trivial result.
\begin{theorem}[{\cite[Thm.~6]{rump-yang16}}]
    Every nonzero ideal in $R$ is a product of maximal and maximal invertible ideals.
\end{theorem}

Of course, it is now natural to wonder how the operations $\circ$ and $+$ on $\bN_0^{(\Max(R))}$ are related.
This is beyond the scope of this article, and leads to left quasirings \cite{rump22}\cite[Proposition 10]{rump25}, an algebraic structure connected to braces, which appear in the study of set-theoretic solutions to the Yang-Baxter equation \cite{vendramin24}.

\section{The Horizon: Open Problems}
We saw how the classical divisor theory of Dedekind domains can be adapted to study factorizations in Dedekind prime rings, and that a similar theory exists for ideals in HNP rings.
However, many basic questions remain open.
For instance, it is not even known whether each $\sL(a)$ is finite in a noetherian domain \cite{bell-brown-nazemian-smertnig23}, although this is an easy consequence of primary decompositions for commutative noetherian domains.

Among Dedekind prime rings, the first Weyl algebra $A_1 \coloneqq \mathbb C\langle x,y \mid yx - xy = 1 \rangle$ is not half-factorial, because
\[
f = y^2x = (1+yx) y
\]
has $\sL(f)=\{2, 3 \}$.
Since $A_1$ is neither bounded nor Hermite, our techniques do not apply.
How wild can the length sets be? In $\cB(\bZ)$ and in the ring of integer-valued polynomials $\operatorname{Int}(\bZ)$ every finite subset of $\bN_{\ge 2}$ occurs as length set \cite{kainrath99,frisch13}.
Might a similar thing be true for the Weyl algebra?

These questions illustrate that, despite the usefulness of divisor and transfer homomorphisms, much is still to be discovered about factorizations in noncommutative rings.

\medskip
\paragraph*{\textbf{Acknowledgements}}{\small
I thank Alfred Geroldinger and my PhD students Mara Pompili and Daniel Vitas for feedback on an earlier version of this manuscript.}

\medskip
\paragraph*{\textbf{Funding Information}}{\small
This work was supported by the Slovenian Research and Innovation Agency (ARIS) under program P1-0288 and grant J1-60025.}

\bibliographystyle{hyperalphaabbr}
\bibliography{expository-ncfact.bib}

\newcommand{\etalchar}[1]{$^{#1}$}
\begin{thebibliography}{BPA{\etalchar{+}}11}

\bibitem[AM53]{asano-murata53}
K.~Asano and K.~Murata.
\newblock Arithmetical ideal theory in semigroups.
\newblock {\em J. Inst. Polytech. Osaka City Univ. Ser. A}, 4:9--33, 1953.

\bibitem[AM69]{atiyah-macdonald69}
M.~F. Atiyah and I.~G. Macdonald.
\newblock {\em Introduction to commutative algebra}.
\newblock Addison-Wesley Publishing Co., Reading, Mass.-London-Don Mills, Ont.,
  1969.

\bibitem[BBNS23]{bell-brown-nazemian-smertnig23}
J.~P. Bell, K.~Brown, Z.~Nazemian, and D.~Smertnig.
\newblock On noncommutative bounded factorization domains and prime rings.
\newblock {\em J. Algebra}, 622:404--449, 2023.
\newblock \href {http://dx.doi.org/10.1016/j.jalgebra.2023.01.023}
  {\path{doi:10.1016/j.jalgebra.2023.01.023}}.

\bibitem[BC11]{baginski-chapman11}
P.~Baginski and S.~T. Chapman.
\newblock Factorizations of algebraic integers, block monoids, and additive
  number theory.
\newblock {\em Amer. Math. Monthly}, 118(10):901--920, 2011.
\newblock \href {http://dx.doi.org/10.4169/amer.math.monthly.118.10.901}
  {\path{doi:10.4169/amer.math.monthly.118.10.901}}.

\bibitem[BC21]{babei-chari21}
A.~Babei and S.~Chari.
\newblock Metacommutation of primes in {E}ichler orders.
\newblock {\em Acta Arith.}, 197(1):77--92, 2021.
\newblock \href {http://dx.doi.org/10.4064/aa191031-23-5}
  {\path{doi:10.4064/aa191031-23-5}}.

\bibitem[BG14]{baeth-geroldinger14}
N.~R. Baeth and A.~Geroldinger.
\newblock Monoids of modules and arithmetic of direct-sum decompositions.
\newblock {\em Pacific J. Math.}, 271(2):257--319, 2014.
\newblock \href {http://dx.doi.org/10.2140/pjm.2014.271.257}
  {\path{doi:10.2140/pjm.2014.271.257}}.

\bibitem[BGGS15]{baeth-geroldinger-grynkiewicz-smertnig15}
N.~R. Baeth, A.~Geroldinger, D.~J. Grynkiewicz, and D.~Smertnig.
\newblock A semigroup-theoretical view of direct-sum decompositions and
  associated combinatorial problems.
\newblock {\em J. Algebra Appl.}, 14(2):1550016, 60, 2015.
\newblock \href {http://dx.doi.org/10.1142/S0219498815500164}
  {\path{doi:10.1142/S0219498815500164}}.

\bibitem[BPA{\etalchar{+}}11]{baeth-ponomarenko-et-al11}
N.~Baeth, V.~Ponomarenko, D.~Adams, R.~Ardila, D.~Hannasch, A.~Kosh,
  H.~McCarthy, and R.~Rosenbaum.
\newblock Number theory of matrix semigroups.
\newblock {\em Linear Algebra Appl.}, 434(3):694--711, 2011.
\newblock \href {http://dx.doi.org/10.1016/j.laa.2010.09.028}
  {\path{doi:10.1016/j.laa.2010.09.028}}.

\bibitem[BW13]{baeth-wiegand13}
N.~R. Baeth and R.~Wiegand.
\newblock Factorization theory and decompositions of modules.
\newblock {\em Amer. Math. Monthly}, 120(1):3--34, 2013.
\newblock \href {http://dx.doi.org/10.4169/amer.math.monthly.120.01.003}
  {\path{doi:10.4169/amer.math.monthly.120.01.003}}.

\bibitem[Car60]{carlitz60}
L.~Carlitz.
\newblock A characterization of algebraic number fields with class number two.
\newblock {\em Proc. Amer. Math. Soc.}, 11:391--392, 1960.
\newblock \href {http://dx.doi.org/10.2307/2034782}
  {\path{doi:10.2307/2034782}}.

\bibitem[Cha19]{chapman19}
S.~T. Chapman.
\newblock So what is class number 2?
\newblock {\em Amer. Math. Monthly}, 126(4):330--339, 2019.
\newblock \href {http://dx.doi.org/10.1080/00029890.2019.1562827}
  {\path{doi:10.1080/00029890.2019.1562827}}.

\bibitem[Cha20]{chari20}
S.~Chari.
\newblock Metacommutation of primes in central simple algebras.
\newblock {\em J. Number Theory}, 206:296--309, 2020.
\newblock \href {http://dx.doi.org/10.1016/j.jnt.2019.06.017}
  {\path{doi:10.1016/j.jnt.2019.06.017}}.

\bibitem[CK15]{cohn-kumar15}
H.~Cohn and A.~Kumar.
\newblock Metacommutation of {H}urwitz primes.
\newblock {\em Proc. Amer. Math. Soc.}, 143(4):1459--1469, 2015.
\newblock \href {http://dx.doi.org/10.1090/S0002-9939-2014-12358-6}
  {\path{doi:10.1090/S0002-9939-2014-12358-6}}.

\bibitem[Cla25]{clark-ca}
P.~C. Clark.
\newblock Commutative algebra.
\newblock Online Lecture Notes, 2025.
\newblock
  \url{https://plclark.github.io/PeteLClark/Expositions/integral2015.pdf},
  retrieved 2026-01-23.

\bibitem[Coh71]{cohn71}
P.~M. Cohn.
\newblock Rings of fractions.
\newblock {\em Amer. Math. Monthly}, 78:596--615, 1971.
\newblock \href {http://dx.doi.org/10.2307/2316568}
  {\path{doi:10.2307/2316568}}.

\bibitem[CR81]{curtis-reiner81}
C.~W. Curtis and I.~Reiner.
\newblock {\em Methods of representation theory. {V}ol. {I}}.
\newblock Pure and Applied Mathematics. John Wiley \& Sons, Inc., New York,
  1981.
\newblock With applications to finite groups and orders, A Wiley-Interscience
  Publication.

\bibitem[CR87]{curtis-reiner87}
C.~W. Curtis and I.~Reiner.
\newblock {\em Methods of representation theory. {V}ol. {II}}.
\newblock Pure and Applied Mathematics (New York). John Wiley \& Sons, Inc.,
  New York, 1987.
\newblock With applications to finite groups and orders, A Wiley-Interscience
  Publication.

\bibitem[CS03]{conway-smith03}
J.~H. Conway and D.~A. Smith.
\newblock {\em On quaternions and octonions: their geometry, arithmetic, and
  symmetry}.
\newblock A K Peters, Ltd., Natick, MA, 2003.

\bibitem[CT24]{cossu-tringali24}
L.~Cossu and S.~Tringali.
\newblock Abstract factorization theorems with applications to idempotent
  factorizations.
\newblock {\em Israel J. Math.}, 263(1):349--395, 2024.
\newblock \href {http://dx.doi.org/10.1007/s11856-024-2623-z}
  {\path{doi:10.1007/s11856-024-2623-z}}.

\bibitem[Erd67]{erdos67}
J.~A. Erdos.
\newblock On products of idempotent matrices.
\newblock {\em Glasgow Math. J.}, 8:118--122, 1967.
\newblock \href {http://dx.doi.org/10.1017/S0017089500000173}
  {\path{doi:10.1017/S0017089500000173}}.

\bibitem[Est91]{estes91}
D.~R. Estes.
\newblock Factorization in hereditary orders.
\newblock {\em Linear Algebra Appl.}, 157:161--164, 1991.
\newblock \href {http://dx.doi.org/10.1016/0024-3795(91)90110-I}
  {\path{doi:10.1016/0024-3795(91)90110-I}}.

\bibitem[Fac02]{facchini02}
A.~Facchini.
\newblock Direct sum decompositions of modules, semilocal endomorphism rings,
  and {K}rull monoids.
\newblock {\em J. Algebra}, 256(1):280--307, 2002.
\newblock \href {http://dx.doi.org/10.1016/S0021-8693(02)00164-3}
  {\path{doi:10.1016/S0021-8693(02)00164-3}}.

\bibitem[Fac12]{facchini12}
A.~Facchini.
\newblock Direct-sum decompositions of modules with semilocal endomorphism
  rings.
\newblock {\em Bull. Math. Sci.}, 2(2):225--279, 2012.
\newblock \href {http://dx.doi.org/10.1007/s13373-012-0024-9}
  {\path{doi:10.1007/s13373-012-0024-9}}.

\bibitem[FGS16]{forsyth-gurev-shrima16}
A.~Forsyth, J.~Gurev, and S.~Shrima.
\newblock Metacommutation as a group action on the projective line over {$\Bbb
  F_p$}.
\newblock {\em Proc. Amer. Math. Soc.}, 144(11):4583--4590, 2016.
\newblock \href {http://dx.doi.org/10.1090/proc/13126}
  {\path{doi:10.1090/proc/13126}}.

\bibitem[FH00]{facchini-herbera00}
A.~Facchini and D.~Herbera.
\newblock {$K_0$} of a semilocal ring.
\newblock {\em J. Algebra}, 225(1):47--69, 2000.
\newblock \href {http://dx.doi.org/10.1006/jabr.1999.8092}
  {\path{doi:10.1006/jabr.1999.8092}}.

\bibitem[Fri13]{frisch13}
S.~Frisch.
\newblock A construction of integer-valued polynomials with prescribed sets of
  lengths of factorizations.
\newblock {\em Monatsh. Math.}, 171(3-4):341--350, 2013.
\newblock \href {http://dx.doi.org/10.1007/s00605-013-0508-z}
  {\path{doi:10.1007/s00605-013-0508-z}}.

\bibitem[FW04]{facchini-wiegand04}
A.~Facchini and R.~Wiegand.
\newblock Direct-sum decompositions of modules with semilocal endomorphism
  rings.
\newblock {\em J. Algebra}, 274(2):689--707, 2004.
\newblock \href {http://dx.doi.org/10.1016/j.jalgebra.2003.06.004}
  {\path{doi:10.1016/j.jalgebra.2003.06.004}}.

\bibitem[Ger88]{geroldinger88}
A.~Geroldinger.
\newblock \"{U}ber nicht-eindeutige {Z}erlegungen in irreduzible {E}lemente.
\newblock {\em Math. Z.}, 197(4):505--529, 1988.
\newblock \href {http://dx.doi.org/10.1007/BF01159809}
  {\path{doi:10.1007/BF01159809}}.

\bibitem[Ger09]{geroldinger09}
A.~Geroldinger.
\newblock Additive group theory and non-unique factorizations.
\newblock In {\em Combinatorial number theory and additive group theory}, Adv.
  Courses Math. CRM Barcelona, pages 1--86. Birkh\"auser Verlag, Basel, 2009.
\newblock \href {http://dx.doi.org/10.1007/978-3-7643-8962-8}
  {\path{doi:10.1007/978-3-7643-8962-8}}.

\bibitem[Ger16]{geroldinger16}
A.~Geroldinger.
\newblock Sets of lengths.
\newblock {\em Amer. Math. Monthly}, 123(10):960--988, 2016.
\newblock \href {http://dx.doi.org/10.4169/amer.math.monthly.123.10.960}
  {\path{doi:10.4169/amer.math.monthly.123.10.960}}.

\bibitem[GHK06]{geroldinger-halter-koch06}
A.~Geroldinger and F.~Halter-Koch.
\newblock {\em Non-unique factorizations}, volume 278 of {\em Pure and Applied
  Mathematics (Boca Raton)}.
\newblock Chapman \& Hall/CRC, Boca Raton, FL, 2006.
\newblock Algebraic, combinatorial and analytic theory.
\newblock \href {http://dx.doi.org/10.1201/9781420003208}
  {\path{doi:10.1201/9781420003208}}.

\bibitem[Gry13]{grynkiewicz13}
D.~J. Grynkiewicz.
\newblock {\em Structural additive theory}, volume~30 of {\em Developments in
  Mathematics}.
\newblock Springer, Cham, 2013.
\newblock \href {http://dx.doi.org/10.1007/978-3-319-00416-7}
  {\path{doi:10.1007/978-3-319-00416-7}}.

\bibitem[Gry22]{grynkiewicz22}
D.~J. Grynkiewicz.
\newblock {\em The characterization of finite elasticities---factorization
  theory in {K}rull monoids via convex geometry}, volume 2316 of {\em Lecture
  Notes in Mathematics}.
\newblock Springer, Cham, 2022.
\newblock \href {http://dx.doi.org/10.1007/978-3-031-14869-9}
  {\path{doi:10.1007/978-3-031-14869-9}}.

\bibitem[GS19]{geroldinger-schmid19}
A.~Geroldinger and W.~A. Schmid.
\newblock A characterization of class groups via sets of lengths.
\newblock {\em J. Korean Math. Soc.}, 56(4):869--915, 2019.
\newblock \href {http://dx.doi.org/10.4134/JKMS.j180467}
  {\path{doi:10.4134/JKMS.j180467}}.

\bibitem[GS23]{geroldinger-schmid23}
A.~Geroldinger and W.~A. Schmid.
\newblock On the incomparability of systems of sets of lengths.
\newblock {\em European J. Combin.}, 111:Paper No. 103694, 25, 2023.
\newblock With a preface by Alain Plagne.
\newblock \href {http://dx.doi.org/10.1016/j.ejc.2023.103694}
  {\path{doi:10.1016/j.ejc.2023.103694}}.

\bibitem[GZ17]{geroldinger-zhong17}
A.~Geroldinger and Q.~Zhong.
\newblock A characterization of class groups via sets of lengths {II}.
\newblock {\em J. Th\'{e}or. Nombres Bordeaux}, 29(2):327--346, 2017.
\newblock \href {http://dx.doi.org/10.5802/jtnb.983}
  {\path{doi:10.5802/jtnb.983}}.

\bibitem[GZ20]{geroldinger-zhong20}
A.~Geroldinger and Q.~Zhong.
\newblock Factorization theory in commutative monoids.
\newblock {\em Semigroup Forum}, 100(1):22--51, 2020.
\newblock \href {http://dx.doi.org/10.1007/s00233-019-10079-0}
  {\path{doi:10.1007/s00233-019-10079-0}}.

\bibitem[Kai99]{kainrath99}
F.~Kainrath.
\newblock Factorization in {K}rull monoids with infinite class group.
\newblock {\em Colloq. Math.}, 80(1):23--30, 1999.
\newblock \href {http://dx.doi.org/10.4064/cm-80-1-23-30}
  {\path{doi:10.4064/cm-80-1-23-30}}.

\bibitem[Lam99]{lam99}
T.~Y. Lam.
\newblock {\em Lectures on modules and rings}, volume 189 of {\em Graduate
  Texts in Mathematics}.
\newblock Springer-Verlag, New York, 1999.
\newblock \href {http://dx.doi.org/10.1007/978-1-4612-0525-8}
  {\path{doi:10.1007/978-1-4612-0525-8}}.

\bibitem[Lev00]{levy00}
L.~S. Levy.
\newblock Modules over hereditary {N}oetherian prime rings (survey).
\newblock In {\em Algebra and its applications ({A}thens, {OH}, 1999)}, volume
  259 of {\em Contemp. Math.}, pages 353--370. Amer. Math. Soc., Providence,
  RI, 2000.
\newblock \href {http://dx.doi.org/10.1090/conm/259/04107}
  {\path{doi:10.1090/conm/259/04107}}.

\bibitem[LM25]{leite-machiavelo25}
A.~Leite and A.~Machiavelo.
\newblock On the {Cycle} {Structure} of the {Metacommutation} {Map}.
\newblock Preprint,
  \href{https://arxiv.org/abs/2504.08709}{{arXiv}:2504.08709}, 2025.

\bibitem[LR11]{levy-robson11}
L.~S. Levy and J.~C. Robson.
\newblock {\em Hereditary {N}oetherian prime rings and idealizers}, volume 174
  of {\em Mathematical Surveys and Monographs}.
\newblock American Mathematical Society, Providence, RI, 2011.
\newblock \href {http://dx.doi.org/10.1090/surv/174}
  {\path{doi:10.1090/surv/174}}.

\bibitem[MR01]{mcconnell-robson01}
J.~C. McConnell and J.~C. Robson.
\newblock {\em Noncommutative {N}oetherian rings}, volume~30 of {\em Graduate
  Studies in Mathematics}.
\newblock American Mathematical Society, Providence, RI, revised edition, 2001.
\newblock With the cooperation of L. W. Small.
\newblock \href {http://dx.doi.org/10.1090/gsm/030}
  {\path{doi:10.1090/gsm/030}}.

\bibitem[Nar04]{narkiewicz04}
W.~Narkiewicz.
\newblock {\em Elementary and analytic theory of algebraic numbers}.
\newblock Springer Monographs in Mathematics. Springer-Verlag, Berlin, third
  edition, 2004.
\newblock \href {http://dx.doi.org/10.1007/978-3-662-07001-7}
  {\path{doi:10.1007/978-3-662-07001-7}}.

\bibitem[Rei75]{reiner75}
I.~Reiner.
\newblock {\em Maximal orders}.
\newblock London Mathematical Society Monographs, No. 5. Academic Press
  [Harcourt Brace Jovanovich, Publishers], London-New York, 1975.

\bibitem[Rot17]{rotman17}
J.~J. Rotman.
\newblock {\em Advanced modern algebra. {P}art 2}, volume 180 of {\em Graduate
  Studies in Mathematics}.
\newblock American Mathematical Society, Providence, RI, third edition, 2017.
\newblock With a foreword by Bruce Reznick.
\newblock \href {http://dx.doi.org/10.1090/gsm/180}
  {\path{doi:10.1090/gsm/180}}.

\bibitem[Rum22]{rump22}
W.~Rump.
\newblock Bijective 1-cocycles, braces, and non-commutative prime
  factorization.
\newblock {\em Colloq. Math.}, 170(1):145--170, 2022.
\newblock \href {http://dx.doi.org/10.4064/cm8684-2-2022}
  {\path{doi:10.4064/cm8684-2-2022}}.

\bibitem[Rum25]{rump25}
W.~Rump.
\newblock The role of divisors in noncommutative ideal theory.
\newblock In {\em Recent progress in ring and factorization theory}, volume 477
  of {\em Springer Proc. Math. Stat.}, pages 391--416. Springer, Cham, 2025.
\newblock \href {http://dx.doi.org/10.1007/978-3-031-75326-8_18}
  {\path{doi:10.1007/978-3-031-75326-8_18}}.

\bibitem[RY16]{rump-yang16}
W.~Rump and Y.~Yang.
\newblock Hereditary arithmetics.
\newblock {\em J. Algebra}, 468:214--252, 2016.
\newblock \href {http://dx.doi.org/10.1016/j.jalgebra.2016.08.015}
  {\path{doi:10.1016/j.jalgebra.2016.08.015}}.

\bibitem[Sch16]{schmid16}
W.~A. Schmid.
\newblock Some recent results and open problems on sets of lengths of {K}rull
  monoids with finite class group.
\newblock In {\em Multiplicative ideal theory and factorization theory}, volume
  170 of {\em Springer Proc. Math. Stat.}, pages 323--352. Springer, [Cham],
  2016.
\newblock \href {http://dx.doi.org/10.1007/978-3-319-38855-7_14}
  {\path{doi:10.1007/978-3-319-38855-7_14}}.

\bibitem[Sme13]{smertnig13}
D.~Smertnig.
\newblock Sets of lengths in maximal orders in central simple algebras.
\newblock {\em J. Algebra}, 390:1--43, 2013.
\newblock \href {http://dx.doi.org/10.1016/j.jalgebra.2013.05.016}
  {\path{doi:10.1016/j.jalgebra.2013.05.016}}.

\bibitem[Sme19]{smertnig19}
D.~Smertnig.
\newblock Factorizations in bounded hereditary {N}oetherian prime rings.
\newblock {\em Proc. Edinb. Math. Soc. (2)}, 62(2):395--442, 2019.
\newblock \href {http://dx.doi.org/10.1017/s0013091518000305}
  {\path{doi:10.1017/s0013091518000305}}.

\bibitem[SV19]{smertnig-voight19}
D.~Smertnig and J.~Voight.
\newblock Definite orders with locally free cancellation.
\newblock {\em Trans. London Math. Soc.}, 6(1):53--86, 2019.
\newblock \href {http://dx.doi.org/10.1112/tlm3.12019}
  {\path{doi:10.1112/tlm3.12019}}.

\bibitem[TV06]{tao-vu06}
T.~Tao and V.~Vu.
\newblock {\em Additive combinatorics}, volume 105 of {\em Cambridge Studies in
  Advanced Mathematics}.
\newblock Cambridge University Press, Cambridge, 2006.
\newblock \href {http://dx.doi.org/10.1017/CBO9780511755149}
  {\path{doi:10.1017/CBO9780511755149}}.

\bibitem[Ven24]{vendramin24}
L.~Vendramin.
\newblock Skew braces: a brief survey.
\newblock In {\em Geometric methods in physics XL, workshop, Bia{\l}owie\.za,
  Poland, June 20--25, 2023}, pages 153--175. Cham: Birkh{\"a}user, 2024.
\newblock \href {http://dx.doi.org/10.1007/978-3-031-62407-0_12}
  {\path{doi:10.1007/978-3-031-62407-0_12}}.

\bibitem[ZS75]{zariski-samuel58}
O.~Zariski and P.~Samuel.
\newblock {\em Commutative algebra. {V}ol. 1}.
\newblock Graduate Texts in Mathematics, No. 28. Springer-Verlag, New
  York-Heidelberg-Berlin, 1958 edition, 1975.
\newblock With the cooperation of I. S. Cohen.

\end{thebibliography}

\end{document}